\documentclass[reqno]{amsart}
\usepackage{amsmath, pb-diagram}
\usepackage{amssymb}
\usepackage{amscd}
\usepackage[all]{xy}

\begin{document}
\title [Rickart Modules Relative to  Goldie Torsion Theory]
{Rickart Modules Relative to  Goldie Torsion Theory}

\author{Burcu Ungor}
\address{Burcu Ungor, Department of Mathematics, Ankara University, Turkey}
\email{bungor@science.ankara.edu.tr}

\author{Sait Hal\i c\i oglu}
\address{Sait Hal\i c\i oglu,  Department of Mathematics, Ankara University, Turkey}
\email{halici@ankara.edu.tr}

\author{Abdullah Harmanci}
\address{Abdullah Harmanci, Department of Mathematics, Hacettepe University, Turkey}
\email{harmanci@hacettepe.edu.tr}

\date{}

\newtheorem {thm}{Theorem}[section]
\newtheorem{lem}[thm]{Lemma}
\newtheorem{prop}[thm]{Proposition}
\newtheorem{cor}[thm]{Corollary}
\newtheorem{df}[thm]{Definition}
\newtheorem{nota}{Notation}
\newtheorem{note}[thm]{Remark}
\newtheorem{ex}[thm]{Example}
\newtheorem{exs}[thm]{Examples}
\newtheorem{rem}[thm]{Remark}
\newtheorem{quo}[thm]{Question}

\begin{abstract}  Let $R$ be an arbitrary ring with
identity and $M$ a right $R$-module with $S=$ End$_R(M)$. Let
$Z_2(M)$ be the second singular submodule of $M$. In this paper,
we define Goldie Rickart modules by utilizing the endomorphisms of
a module. The module $M$ is called Goldie Rickart if for any $f\in
S$, $f^{-1}(Z_2(M))$ is a direct summand of $M$. We provide
several characterizations of Goldie Rickart modules and study
their properties.  Also we present that semisimple rings and right
$\Sigma$-$t$-extending rings admit some characterizations in terms
of Goldie Rickart modules.
\\[+2mm]
 \noindent {\bf 2010 MSC:}  13C99, 16D80, 16U80.

\noindent {\bf Key words:} Rickart module, Goldie Rickart module.
\end{abstract}

\maketitle
\section{ Introduction }
Throughout this paper $R$ denotes a ring with identity, modules
are unital right $R$-modules. Let $M$ be an $R$-module with $S=$
End$_R(M)$. The singular submodule of $M$ is
 $Z(M) = \{m\in M\mid mI = 0$ for some essential
right ideal $I$ of $R$ $\}$.  If $M=Z(M)$, then $M$ is called {\it
singular} and $M$ is {\it nonsingular} provided $Z(M) = 0$. The
second singular submodule,  in other words, the Goldie torsion
submodule  $Z_2(M)$ of $M$ is defined by $Z(M/Z(M)) =
Z_2(M)/Z(M)$. The module $M$ is called {\it $Z_2$-torsion} (or
{\it Goldie torsion}) if $M = Z_2(M)$. It is evident that every
singular module is $Z_2$-torsion.   If for any $f\in S$, $r_M(f)$
is a direct summand of $M$, then $M$ is said to be a {\it Rickart
module}. Rickart modules are introduced and investigated by Lee,
Rizvi and Roman in \cite{LRR}. Also, right (left) Rickart rings
(or principally projective rings) initially appeared in Maeda
\cite{M}, and were further studied by Hattori \cite{Hat}, that is,
a ring is called {\it right (left) Rickart} if every principal
right (left)  ideal is projective, equivalently, the right
annihilator of any single element is generated by  an idempotent
as a right ideal. The concept of right (left) Rickart rings has
been comprehensively studied in the literature. In \cite{AH},
Asgari and Haghany defined {\it t-Baer modules}, that is, a module
$M$ is called {\it t-Baer} if $t_M(I)$ is a direct summand of $M$
for every left ideal $I$ of $S$ where $t_M(I)=\{m\in M \mid Im\leq
Z_2(M)\}$, and they study properties of t-Baer modules. Note that
for any $f\in S$, $f^{-1}(Z_2(M)) = t_M(Sf)$ and $\bigcap
\limits_{f\in S} f^{-1}(Z_2(M))=Z_2(M)$. Clearly, the kernel
$r_M(f)$ of $f\in S$ is a submodule of $f^{-1}(Z_2(M))$.

In what follows, by $\Bbb Z$, $\Bbb Q$ and $\Bbb Z_n$ we denote,
respectively, integers, rational numbers and  the ring of integers
modulo $n$. For a module $M$, $E(M)$ is the injective hull of $M$
and $S=$ End$_R(M)$ is the ring of endomorphisms of $M$.

 \indent\indent

\section{Goldie Rickart Modules}

In this section we give our main definition, namely Goldie Rickart
modules, and investigate some properties of this class of modules.

\begin{df} A module $M$ is called {\it Goldie Rickart} if $f^{-1}(Z_2(M))$ is a direct summand of
$M$ for every $f\in S$.
\end{df}

It is clear that every semisimple module, every singular module
and every $Z_2$-torsion module is Goldie Rickart. If $R$ is a ring
with $Z_2(R_R)=R$, then every $R$-module $M$ is Goldie Rickart due
to $Z_2(M)=M$. For a nonsingular module $M$, since
$r_M(f)=f^{-1}(Z_2(M))$, $M$ is  Rickart if and only if it is
Goldie Rickart, however these two notions are not equivalent for
any arbitrary module that will be shown later.
\begin{prop} Every indecomposable Goldie Rickart module is
a Rickart or $Z_2$-torsion module.
\end{prop}
\begin{proof} Let $M$ be an indecomposable Goldie Rickart module
and $1$ denote the identity endomorphism of $M$. Since
$1^{-1}(Z_2(M))=Z_2(M)$, $Z_2(M)=M$ or $Z_2(M)=0$. This implies
that $M$ is $Z_2$-torsion or it is nonsingular and so it is
Rickart.
\end{proof}

\begin{prop} Every indecomposable extending module is a nonsingular
or Goldie Rickart module.
\end{prop}
\begin{proof} Let $M$ be an indecomposable extending module. $Z_2(M)$ is a direct summand of $M$
since it is a closed submodule of $M$. Hence $Z_2(M)=0$ or
$Z_2(M)=M$. This implies that $M$ is nonsingular or Goldie Rickart
due to $f^{-1}(Z_2(M))=M$ for every $f\in S$.
\end{proof}

Recall that a module $M$ has the {\it (strong) summand
intersection property} if the intersection of (any) two direct
summands is a direct summand of $M$.

\begin{prop} Every $t$-Baer module $M$ is Goldie Rickart. The
converse holds if $M$ has the strong summand intersection property
for direct summands which contain $Z_2(M)$.
\end{prop}
\begin{proof} Let $M$ be a $t$-Baer module and $f\in S$. Since
$t_M(Sf)=f^{-1}(Z_2(M))$, $M$ is a Goldie Rickart module. The
converse is true due to \cite[Theorem 3.2]{AH}.
\end{proof}

In \cite{AH} it is said that a submodule $N$ of a module $M$ is
{\it t-essential}  if for every submodule $L$ of $M$, $N\cap L\leq
Z_2(M)$ implies that $L\leq Z_2(M)$, and $N$ is called {\it
$t$-closed} if $N$ has no t-essential extension in $M$. The module
$M$ is called {\it $t$-extending} if every $t$-closed submodule of
$M$ is a direct summand of $M$, while  a ring $R$ is called {\it
right $\Sigma$-$t$-extending} if every free $R$-module is
$t$-extending.

\begin{prop}\label{t-ext} Every $t$-extending module is Goldie Rickart.
\end{prop}

\begin{proof} Let $M$ be a $t$-extending module and $f\in S$.
By \cite[Corollary 2.7]{AH}, for any module $M$, $f^{-1}(Z_2(M))$
is $t$-closed in $M$. By hypothesis $f^{-1}(Z_2(M))$ is a direct
summand of $M$.
\end{proof}

By combining Proposition \ref{t-ext} with \cite[Theorem 3.12]{AH} we have Theorem \ref{dortlu}.
\begin{thm}\label{dortlu} The following are equivalent for a ring $R$.
\begin{enumerate}
\item[{\rm (1)}] $R$ is right $\Sigma$-$t$-extending.
\item[{\rm (2)}] Every $R$-module is $t$-extending.
\item[{\rm (3)}] Every $R$-module is $t$-Baer.
\item[{\rm (4)}] Every $R$-module is Goldie Rickart.
\end{enumerate}
\end{thm}

We obtain the next result as an immediate consequence of
\cite[Theorem 2.15]{G}, \cite[Theorem 3.12]{AH} and Theorem
\ref{dortlu}.
\begin{prop} If a ring $R$ is Morita-equivalent to a finite direct product of full lower
triangular matrix rings over division rings, then every $R$-module
is Goldie Rickart.
\end{prop}

We now give a useful characterization of Goldie Rickart modules by
using Goldie torsion submodules.

\begin{thm}\label{goldie} A module $M$ is Goldie Rickart if and only if $M=Z_2(M)\oplus
N$ where $N$ is a (nonsingular) Rickart module.
\end{thm}

\begin{proof} Let $M$ be a Goldie Rickart module and $1_M$
denote the identity endomorphism of $M$. Then
$1_M^{-1}(Z_2(M))=Z_2(M)$ is a direct summand of $M$. Let
$M=Z_2(M)\oplus N$ for some submodule $N$ of $M$ and $f\in $
End$_R(N)$. Hence $1_{Z_2(M)}\oplus f \in S$, say
$g=1_{Z_2(M)}\oplus f$. This implies that
$g^{-1}(Z_2(M))=Z_2(M)\oplus r_N(f)$. By assumption,
$g^{-1}(Z_2(M))$ is a direct summand of $M$. It follows that
$r_N(f)$ is a direct summand of $N$. Therefore $N$ is Rickart.
Since $M/Z_2(M)$ is nonsingular, $N$ is a nonsingular module. For
the converse, assume that $M=Z_2(M)\oplus N$ where $N$ is a
(nonsingular) Rickart module. Let $f\in S$ and $\pi _N$ denote the
projection on $N$ along $Z_2(M)$. Then $\pi _N f_{|_N} \in $
End$_R(N)$ and it can be easily shown that
$f^{-1}(Z_2(M))=Z_2(M)\oplus r_N(\pi _N f_{|_N})$. Since $N$ is
Rickart, $r_N(\pi _N f_{|_N})$ is a direct summand of $N$, and so
$f^{-1}(Z_2(M))$ is a direct summand of $M$. This completes the
proof.
\end{proof}

In \cite{NZ}, if $I$ is an ideal of a ring $R$, it is said that
{\it idempotents lift strongly modulo $I$} if whenever $a^2-a\in
I$, there exists $e^2 = e\in aR$ (equivalently $e^2 = e\in  Ra$)
such that $e-a\in I$. Also a ring $R$ is  called {\it
$Z_2(R_R)$-semiperfect} \cite{NZ} if $R/Z_2(R_R)$ is semisimple
and idempotents lift strongly modulo $Z_2(R_R)$.

\begin{cor}\label{semiper} Let $R$ be a $Z_2(R_R)$-semiperfect ring. Then every $R$-module  is
Goldie Rickart.
\end{cor}
\begin{proof} Let $M$ be an $R$-module. Then $M=Z_2(M)\oplus N$ where $N$ is semisimple by
\cite[Theorem 49]{NZ}. Hence $M$ is Goldie Rickart  in virtue of
Theorem \ref{goldie}.
\end{proof}

In the light of \cite[Theorem 2.5]{YZ}, if  $R$ is a QF-ring, then
it is $Z_2(R_R)$-semiperfect. Then we have the next result due to
Corollary \ref{semiper}.
\begin{cor} Every module over a QF-ring is Goldie Rickart.
\end{cor}

Due to Theorem \ref{goldie}, if $M$ is a Goldie Rickart module,
then $M/Z_2(M)$ is a Rickart module, and it is Goldie Rickart
since it is nonsingular. But the converse does not hold in
general, as the following example shows.

\begin{ex}\label{ornek}\rm{ Let $\mathcal{P}=\{p\in \Bbb Z \mid
p~\mbox{is~prime}\}$ and consider the $\Bbb Z$-module
$M=\prod\limits_{p\in \mathcal{P}}\Bbb Z_p$. Then
$Z(M)=\bigoplus\limits_{p\in \mathcal{P}}\Bbb Z_p$ and
$Z_2(M)=Z(M)$. Since $M/Z_2(M)$ is extending and nonsingular, it
is Rickart. But $Z_2(M)$ is not a direct summand of $M$, and so
$M$ is not Goldie Rickart, by Theorem \ref{goldie}. }
\end{ex}

\begin{rem}\rm{ Example \ref{ornek} also reveals the fact that
if $M/N$ is a Goldie Rickart module for any submodule $N$ of a
module $M$, then $M$ need not be Goldie Rickart. Because in
Example \ref{ornek}, the module $M/Z_2(M)$  is Rickart and
nonsingular, hence it is Goldie Rickart while $M$ is not Goldie
Rickart. }
\end{rem}

According to next examples, Rickart modules and Goldie Rickart
modules do not imply each other.

\begin{exs}\label{ex}\rm{
(1) Consider the module $M$ in the Example \ref{ornek}. It is
known from there $M$ is not Goldie Rickart. On the other hand, the
endomorphism ring $S$ of $M$ is  $\prod\limits_{p\in
\mathcal{P}}\Bbb Z_p$. Since $S$ is a von Neumann regular ring,
$M$ is Rickart by \cite[Corollary 3.2]{Wa}.

(2) Consider $\Bbb Z_4$ as a $\Bbb Z$-module. Then $\Bbb Z_4$
is a Goldie Rickart module due to $Z(\Bbb Z_4)=\Bbb Z_4=Z_2(\Bbb
Z_4)$. On the other hand, for $f\in $ End$_\Bbb Z(\Bbb Z_4)$ with
$f(\overline{1})=\overline{2}$,  $r_{\Bbb Z_4}(f)=2\Bbb Z_4$ is
not a direct summand of $\Bbb Z_4$. Hence $\Bbb Z_4$ is not
Rickart. }
\end{exs}

Now we give a relation between the Rickart and Goldie Rickart
modules.

\begin{thm}\label{versus} Let $M$ be a module. Then the following are
equivalent.
\begin{enumerate}
    \item[{\rm (1)}] $M$ is Goldie Rickart and $r_M(f)$ is a direct summand of $f^{-1}(Z_2(M))$
for any $f\in S$.
   \item[{\rm (2)}] $M$ is Rickart and $Z_2(M)$ is a
direct summand of $M$.
\end{enumerate}
\end{thm}

\begin{proof} (1) $\Rightarrow$ (2) Let $M$ be a Goldie Rickart module and $f\in S$.
Then $f^{-1}(Z_2(M))$ is a direct summand of $M$ and by
hypothesis, $r_M(f)$ is a direct summand of $f^{-1}(Z_2(M))$. It
follows that $M$ is Rickart. In addition, by Theorem \ref{goldie},
$Z_2(M)$ is a direct summand of $M$.

(2) $\Rightarrow$ (1) Let $M$ be a Rickart module and
$M=Z_2(M)\oplus N$ for some submodule $N$ of $M$. Then $N$ is
Rickart and so $M$ is Goldie Rickart by Theorem \ref{goldie}. The
rest is clear since $M$ is Rickart and $r_M(f)$ is a submodule of
$f^{-1}(Z_2(M))$ for any $f\in S$.
\end{proof}

\begin{prop} The following hold for a module $M$.
\begin{enumerate}
    \item [{\rm (1)}] If $M$ is Rickart with $Z(M)$ a direct
summand of $M$, then $M$ is Goldie Rickart.
    \item [{\rm (2)}] If $M$ is Goldie Rickart and $R$ is right
nonsingular, then $Z(M)$ is a direct summand of $M$.
\end{enumerate}
\end{prop}
\begin{proof} (1) Let $Z(M)$ be a direct summand of $M$. Since $Z(M)$ is essential in $Z_2(M)$,
we have $Z(M)=Z_2(M)$.  Then Theorem \ref{versus} completes the
proof.

(2) The right nonsingularity of $R$ implies that $Z(M)=Z_2(M)$.
The rest is clear because $M$ is Goldie Rickart.
\end{proof}

\begin{lem}\label{sequence} Let $M$ be a module. Then the following are
equivalent.
\begin{enumerate}
\item[{\rm (1)}] $M$ is a Goldie Rickart module.
\item[{\rm (2)}] The exact sequence
$0\rightarrow f^{-1}(Z_2(M))\rightarrow M\rightarrow
M/f^{-1}(Z_2(M))\rightarrow 0$ is split for any $f\in S$.
\end{enumerate}
\end{lem}
\begin{proof} Let $f\in S$ and consider the exact sequence $0\rightarrow f^{-1}(Z_2(M))\rightarrow M\rightarrow M/f^{-1}(Z_2(M))\rightarrow
    0$. Then $M$ is Goldie Rickart if and only if $f^{-1}(Z_2(M))$ is a
    direct summand of $M$ if and only if the exact sequence is
    split.
\end{proof}

In the next result we give a characterization of semisimple rings
by using the notion of Goldie Rickart modules.

\begin{thm} The following are equivalent for a ring  $R$.
\begin{enumerate}
    \item[{\rm (1)}] Every $R$-module is Goldie Rickart and its Goldie torsion submodule  is projective.
    \item[{\rm (2)}] $R$ is semisimple.
\end{enumerate}
\end{thm}
\begin{proof}  (1) $\Rightarrow$ (2) Let $M$ be an $R$-module.
Since $M$ is Goldie Rickart, by Theorem \ref{goldie},
$M=Z_2(M)\oplus N$ for some nonsingular submodule $N$ of $M$. By
hypothesis $Z_2(M)$ is projective. Also by Theorem \ref{dortlu}
and \cite[Theorem 3.12]{AH}, $N$ is projective. Hence $M$ is
projective and so $R$ is semisimple due to \cite[Corollary
17.4]{AF}.

(2) $\Rightarrow$ (1) Let $M$ be an $R$-module and $f\in S$.
Consider the exact sequence $0\rightarrow
f^{-1}(Z_2(M))\rightarrow M\rightarrow M/f^{-1}(Z_2(M))\rightarrow
0$. By \cite[Proposition 13.9]{AF}, this exact sequence is split,
and so $M$ is Goldie Rickart due to Lemma \ref{sequence}. The rest
is clear from \cite[Corollary 17.4]{AF}.
\end{proof}

Recall that a  module $M$ is called {\it duo} if every submodule
of $M$ is fully invariant, i.e., for a submodule $N$ of $M$,
$f(N)\leq N$ for each $f\in S$.  Fully invariant submodules of
Goldie Rickart modules are also Goldie Rickart under some
conditions.

\begin{lem}\label{duo} Let $M$ be a Goldie Rickart module  and $N$ a fully invariant submodule of $M$.
If every endomorphism of $N$ can be extended to an endomorphism of
$M$, then $N$ is Goldie Rickart.
\end{lem}
\begin{proof} Let $f\in $ End$_R(N)$. By hypothesis, there exists
$g\in S$ such that  $g_{\mid_N}=f$ and being $M$ Goldie Rickart,
there exists $e=e^2\in S$ such that $g^{-1}(Z_2(M))=eM$. Since $N$
is fully invariant, $e_{\mid_N}$ is an idempotent of End$_R(N)$.
We claim that $f^{-1}(Z_2(N))=e_{\mid_N}N$. Clearly,
$f^{-1}(Z_2(N))\subseteq e_{\mid_N}N$. In order to see other
inclusion, let $n\in N$. Then $fe_{\mid_N}n=gen\in Z_2(M)\cap
N=Z_2(N)$, and so $e_{\mid_N}n\in f^{-1}(Z_2(N))$. Hence we have
$e_{\mid_N}N\subseteq f^{-1}(Z_2(N))$. This implies that  $N$ is
Goldie Rickart.
\end{proof}

Recall that a module $M$ is called {\it quasi-injective} if  it is
$M$-injective. It is well known that every quasi-injective module
is a fully invariant submodule of its injective hull. By
considering this fact, we can say the next result as an immediate
consequence of Lemma \ref{duo}.

\begin{prop} Let $M$ be a quasi-injective module. If $E(M)$ is Goldie Rickart, then so is $M$.
\end{prop}

\begin{thm}\label{quasi-inj duo} Let $M$ be a quasi-injective duo module.  If $M$ is Goldie Rickart, then so is every submodule of
$M$.
\end{thm}
\begin{proof} Let $M$ be a Goldie Rickart module, $N$ a submodule of $M$  and  $f\in$ End$_R(N)$.
By quasi-injectivity of $M$, $f$ extends to an endomorphism $g$ of
$M$. Then $g^{-1}(Z_2(M)) = eM$ for some $e=e^2 \in S$. Since $N$
is fully invariant in $M$, the proof follows from Lemma \ref{duo}.
\end{proof}

\begin{prop}\label{dik toplan} Every direct summand of a Goldie Rickart module is
also Goldie Rickart.
\end{prop}

\begin{proof} Let $M$ be a Goldie Rickart module and $N$ a direct
summand of $M$. There exists a submodule $K$ of $M$ with
$M=N\oplus K$. Let $f\in $ End$_R(N)$. Hence $f\oplus 1_K \in S$,
say $g=f\oplus 1_K$. Since $M$ is Goldie Rickart, $g^{-1}(Z_2(M))$
is a direct summand of $M$. On the other hand,
$g^{-1}(Z_2(M))=f^{-1}(Z_2(N))\oplus Z_2(K)$. Thus
$f^{-1}(Z_2(N))$ is a direct summand of $N$. Therefore $N$ is
Goldie Rickart.
\end{proof}

In comparison with Proposition \ref{dik toplan}, in general,  a
direct sum of Goldie Rickart modules may not be Goldie Rickart as
shown below.

\begin{ex}\label{direct-sum-ornek}{\rm  Let $R$ denote the ring of $2\times 2$ upper triangular matrices over $\Bbb Q[x]$,
$M$ right $R$-module $R$ and $e_{ij},(1\leq i,j\leq 2)$, $2\times
2$ matrix units. Consider the submodules $N = e_{11}R$ and $K =
e_{22}R$ of $M$. Then $M = N\oplus K$. Note that $M$ is a
nonsingular module and $S\cong R$. It is evident that $N$ and $K$
are Goldie Rickart. On the other hand, for $f=e_{11}2x+e_{12}x\in
S$, $r_M(f)=(-e_{12}x+ 2e_{22}x)R$ is not a direct summand of $M$.
Hence $M$ is not Rickart, therefore it is not Goldie Rickart.}
\end{ex}

Now we investigate some conditions about when direct sums of
Goldie Rickart modules are also Goldie Rickart, but more details
are in the last section.

\begin{prop}\label{degisti} Let $\{M_i\}_{i\in \mathcal{I}}$ be a class of  $R$-modules for an arbitrary index set
$\mathcal{I}$. If Hom$_R(M_i, M_j)=0$ for every $i, j\in
\mathcal{I}$ with $i\neq j$ (i.e., for every  $i\in \mathcal{I}$,
$M_i$ is a fully invariant submodule of $\bigoplus \limits_{i\in
\mathcal{I}} M_i$), then $\bigoplus \limits_{i\in \mathcal{I}}
M_i$ is Goldie Rickart if and only if $M_i$ is Goldie  Rickart for
every $i\in \mathcal{I}$.
\end{prop}
\begin{proof} The necessity is clear by Proposition \ref{dik
toplan}. Conversely, let $M=\bigoplus \limits_{i\in \mathcal{I}}
M_i$ and $f=(f_{ij})\in S$ where $f_{ij}\in $ Hom$_R(M_j, M_i)$.
Then $f_{ii}^{-1}(Z_2(M_i))$ is a direct summand of $M_i$ for each
$i\in \mathcal{I}$. On the other hand, we have
$f^{-1}(Z_2(M))=\bigoplus\limits_{i\in \mathcal{I}}
f_{ii}^{-1}(Z_2(M_i))$. Hence $f^{-1}(Z_2(M))$ is a direct summand
of $M$, as asserted.
\end{proof}

A ring is called {\it abelian} if all its idempotents are central.
A module is called {\it  abelian} if its endomorphism ring is
abelian. It is well known that a module $M$ is abelian if and only
if every direct summand of $M$ is fully invariant in $M$.

\begin{cor}\label{diktoplam} Let $\{M_i\}_{i\in \mathcal{I}}$ be a class of $R$-modules for an arbitrary index set
$\mathcal{I}$ and $\bigoplus\limits_{i\in \mathcal{I}} M_i$ an
abelian module. Then $\bigoplus \limits_{i\in \mathcal{I}} M_i$ is
Goldie Rickart if and only if $M_i$ is Goldie  Rickart for all
$i\in \mathcal{I}$.
\end{cor}

\begin{prop}\label{kendi-dik} Let $M$ be a Goldie Rickart module with its endomorphism ring von
Neumann regular. Then any finite direct sum of copies of $M$ is
also Goldie Rickart.
\end{prop}
\begin{proof} Let $\mathcal{I}$ be a finite index set, assume
$\mathcal{I}=\{1, 2, \dots, n\}$. By Theorem \ref{goldie}, we have
$M=Z_2(M)\oplus N$ where $N$ is Rickart. Then End$_R(N)=eSe$ for
some idempotent $e\in S$, and so
End$_R(\bigoplus\limits_{\mathcal{I}} N)=M_n($End$_R(N))$.  Since
$S$ is von Neumann regular and the von Neumann regularity is
Morita invariant, End$_R(\bigoplus\limits_{\mathcal{I}} N)$ is von
Neumann regular. Hence $\bigoplus\limits_ {\mathcal{I}} N$ is
Rickart by \cite[Corollary 3.2]{Wa}. Also we have
$\bigoplus\limits_{\mathcal{I}} M= \bigoplus\limits_{\mathcal{I}}
Z_2(M)\oplus (\bigoplus\limits_{\mathcal{I}}
N)=Z_2(\bigoplus\limits_{\mathcal{I}} M) \oplus
(\bigoplus\limits_{\mathcal{I}} N)$. This implies that
$\bigoplus\limits_{\mathcal{I}} M$ is Goldie Rickart.
\end{proof}

\begin{lem}\label{SIP} Let $M$ be a Goldie Rickart module and $N$
a direct summand of $M$ which contains $Z_2(M)$. Then for any
direct summand $K$ of $M$, $N\cap K$ is also a direct summand of
$M$.
\end{lem}

\begin{proof} Let $K$ be any direct summand of $M$. Then there
exist idempotents $e, f\in S$ such that $N=eM$ and $K=fM$. Since
$Z_2(M)\subseteq eM$, we have $eM=(1-e)^{-1}(Z_2(M))$. We claim
that $((1-e)f)^{-1}(Z_2(M))=(eM\cap fM)\oplus (1-f)M$. Let $m\in
((1-e)f)^{-1}(Z_2(M))$. Hence $fm\in (1-e)^{-1}(Z_2(M))$ and so
$m=fm+(1-f)m\in (eM\cap fM)\oplus (1-f)M$. This implies that
$((1-e)f)^{-1}(Z_2(M))\subseteq (eM\cap fM)\oplus (1-f)M$. For the
reverse inclusion, let $x+y\in (eM\cap fM)\oplus (1-f)M$. Thus
$(1-e)f(x+y)=(1-e)fx+(1-e)fy=(1-e)x+(1-e)f(1-f)y\in Z_2(M)$. Then
we have $(eM\cap fM)\oplus (1-f)M\subseteq ((1-e)f)^{-1}(Z_2(M))$.
Since $M$ is Goldie Rickart, $((1-e)f)^{-1}(Z_2(M))$ is a direct
summand of $M$. Therefore $eM\cap fM$ is a direct summand of $M$,
as required.
\end{proof}

By virtue of Lemma \ref{SIP}, we obtain the next result, and then
we give another characterization of Goldie Rickart modules.

\begin{prop}\label{sumint} Let $M$ be a Goldie Rickart module. Then $M$ has the
summand intersection property for direct summands which contain
$Z_2(M)$.
\end{prop}

The converse of Proposition \ref{sumint} does not hold in general,
for example the module $M$ in Examples \ref{ex}(1) is Rickart and
so it has the summand intersection property by \cite[Proposition
2.16]{LRR}, but it is not Goldie Rickart.

\begin{thm}\label{finite} The following are equivalent for a module $M$.
\begin{enumerate}
\item [{\rm (1)}] $M$ is Goldie Rickart.
\item [{\rm (2)}] $t_M(I)$ is a direct summand of $M$ for each
finite subset $I$ of $S$.
\end{enumerate}
\end{thm}
\begin{proof} (1) $\Rightarrow$ (2) Let $n\in \Bbb N$ and $I=\{f_1, f_2, \dots , f_n\}\subseteq
S$. For the proof, we apply induction on $n$. If $n=1$, then there
is nothing to show. Now let $n>1$ and suppose the claim holds for
$n-1$. Hence $t_M(J)$ is a direct summand of $M$ where $J=\{f_1,
f_2, \dots , f_{n-1}\}$. Clearly, we have $t_M(I)=t_M(J)\cap
{f_n}^{-1}(Z_2(M))$ and ${f_n}^{-1}(Z_2(M))$ is also a direct
summand of $M$ by (1). Since $t_M(J)$ and ${f_n}^{-1}(Z_2(M))$
contain $Z_2(M)$, by Proposition \ref{sumint}, $t_M(I)$ is a
direct summand of $M$.

(2) $\Rightarrow$ (1) Obvious.
\end{proof}

\begin{prop} Let $M$ be a Goldie Rickart and projective (injective) module.
Then for every direct summand $N$ of $M$, $Z_2(M)+N$ is also a
projective (injective) module.
\end{prop}

\begin{proof} Let $N$ be a direct summand of $M$. By Theorem
\ref{goldie}, $M=Z_2(M)\oplus K$ for some submodule $K$ of $M$.
Then $Z_2(M)\cap N=Z_2(N)$ is also a direct summand of $M$ due to
Lemma \ref{SIP}, let $M=Z_2(N)\oplus L$ for some submodule $L$ of
$M$. Hence $Z_2(M)=Z_2(N)\oplus Z_2(L)$ and $N=Z_2(N)\oplus (N\cap
L)$. It follows that $Z_2(M)+N=Z_2(N)\oplus Z_2(L)\oplus (N\cap
L)$. Since $M$ is projective (injective), $Z_2(N), Z_2(L)$ and
$N\cap L$ is also projective (injective), and so $Z_2(M)+N$ is
projective (injective).
\end{proof}

\begin{lem}\label{Zler} Let $R$ be a ring, $M$  an $R$-module and $N$ a submodule of $M$.
Then $(Z(M)+N)/N\subseteq Z(M/N)$
and $(Z_2(M)+N)/N\subseteq Z_2(M/N)$. Moreover, if $R$ is a ring
without zero divisors and the submodule $N$ is a torsion
$R$-module, then $(Z(M)+N)/N=Z(M/N)$ and $(Z_2(M)+N)/N=Z_2(M/N)$.
\end{lem}
\begin{proof} It is easy to see that if
$x\in Z(M)$, then $x+N \in Z(M/N)$, and if $x\in Z_2(M)$, then
$x+N \in Z_2(M/N)$. Hence $(Z(M)+N)/N\subseteq Z(M/N)$ and
$(Z_2(M)+N)/N\subseteq Z_2(M/N)$. Now let $R$ be a ring without
zero divisors and the submodule $N$ a torsion $R$-module. Then for
any $m\in M$, $r_R(m)$ is essential in $r_R(m+N)$. It follows that
$Z(M/N)\subseteq (Z(M)+N)/N$, and so we have $(Z(M)+N)/N=Z(M/N)$.
This implies that $Z_2(M/N)\subseteq (Z_2(M)+N)/N$. Thus
$(Z_2(M)+N)/N=Z_2(M/N)$.
\end{proof}

\begin{lem}\label{kapsam} Let $M$ be a quasi-projective module
and $N$ a submodule of $M$. Then for each $\overline{f}\in $
End$_R(M/N)$, there exists $f\in S$ such that
$(f^{-1}(Z_2(M))+N)/N\subseteq \overline{f}^{-1}(Z_2(M/N))$.
Moreover, if $R$ is a ring without zero divisors and the submodule
$N$ is a torsion $R$-module, then
$(f^{-1}(Z_2(M))+N)/N=\overline{f}^{-1}(Z_2(M/N))$.
\end{lem}
\begin{proof} Let $\overline{f}\in $ End$_R(M/N)$ and $\pi$ denote
the natural epimorphism from $M$ to $M/N$. Consider the following
diagram
$$\xymatrix{ M \ar@{.>}[d]_f \ar[r]^\pi & M/N  \ar[d]^{\overline{f}} \\
M \ar[r]_\pi & M/N\\
 }$$
 Since $M$ is quasi-projective, there exists $f\in S$ such that $\overline{f}\pi=\pi
 f$. Let $m\in f^{-1}(Z_2(M))$. Since $fm\in Z_2(M)$, by Lemma \ref{Zler},
$\overline{f}(m+N)=\overline{f}\pi (m)=\pi f(m)=fm+N\in Z_2(M/N)$.
Therefore $m+N\in \overline{f}^{-1}(Z_2(M/N))$, and so
$(f^{-1}(Z_2(M))+N)/N\subseteq \overline{f}^{-1}(Z_2(M/N))$. Let
$R$ be a ring without zero divisors and the submodule $N$ a
torsion $R$-module. Then for any $m+N\in
\overline{f}^{-1}(Z_2(M/N))$, by Lemma \ref{Zler}, we have
$\overline{f}(m+N)=fm+N\in (Z_2(M)+N)/N$. Hence $m\in
f^{-1}(Z_2(M))$, and so $\overline{f}^{-1}(Z_2(M/N))\subseteq
(f^{-1}(Z_2(M))+N)/N$.
\end{proof}

\begin{prop} Let $R$ be a ring without zero divisors and $M$ a
quasi-projective Goldie Rickart module. If $N$ is a fully
invariant submodule of $M$ and a torsion $R$-module, then $M/N$ is
also Goldie Rickart.
\end{prop}
\begin{proof} Let $\overline{f}\in $ End$_R(M/N)$. By Lemma
\ref{kapsam}, there exists $f\in S$ with $(f^{-1}(Z_2(M))+N)/N =
\overline{f}^{-1}(Z_2(M/N))$. Since $M$ is Goldie Rickart,
$f^{-1}(Z_2(M))=eM$ for some $e^2=e\in S$. Let $\pi$ denote the
natural epimorphism from $M$ to $M/N$ and consider the following
diagram
$$\xymatrix{ M \ar[d]_e \ar[r]^\pi & M/N  \ar@{.>}[d]^{\overline{e}} \\
M \ar[r]_\pi & M/N\\
 }$$
 Since $N$ is  fully invariant, by the Factor Theorem, there exists a unique homomorphism
 $\overline{e}\in $ End$_R(M/N)$  such that $\overline{e} \pi=\pi
 e$. It follows that $\overline{e}^2=\overline{e}$.
 Also $\overline{e}(M/N)=(f^{-1}(Z_2(M))+N)/N$. This completes the proof.
\end{proof}

\section{Applications : Goldie Rickart Rings}

In this section we study the concept of Goldie Rickart for the
ring case. A ring $R$ is called {\it right Goldie Rickart} if the
right $R$-module $R$ is Goldie Rickart, i.e., for any $a\in R$,
the right ideal $a^{-1}(Z_2(R_R)) = \{b\in R\mid ab\in Z_2(R_R)\}$
is a direct summand of $R$. As a consequence of Theorem
\ref{finite}, a ring $R$ is right Goldie Rickart if and only if
$t_R(I)$ is a direct summand of $R$ as a right ideal for each
finite subset $I$ of $R$.  Left Goldie Rickart rings are defined
similarly. Goldie Rickart rings are not left-right symmetric as
the following example shows.

\begin{ex}\label{left-right}{\rm Consider the ring $R=\left[\begin{array}{cc} \Bbb Z
& \Bbb Z_2 \\0 &\Bbb Z_2\end{array}\right]$ in \cite[(7.22)
Example]{L}. It is shown that $R$ is right nonsingular and
$Z(_RR)=\{0, x\}$ where $x=\left[\begin{array}{cc} 0 &
\overline{1} \\0 &0\end{array}\right]$. Also
$Z\big(_R(R/Z(_RR))\big)=\{ 0, \overline{m}\}$ where
$m=\left[\begin{array}{cc} 0 & 0 \\0
&\overline{1}\end{array}\right]\in R$, and so
$Z_2(_RR)=\left[\begin{array}{cc} 0 & \Bbb Z_2 \\0 &\Bbb
Z_2\end{array}\right]$. Thus $R=Z_2(_RR)\oplus
\left[\begin{array}{cc} \Bbb Z & 0 \\0 &0\end{array}\right]$. It
can be easily shown that $\left[\begin{array}{cc} \Bbb Z & 0 \\0
&0\end{array}\right]$  is a Rickart left $R$-module. Therefore $R$
is a left Goldie Rickart ring by Theorem \ref{goldie}. On the
other hand, for $y=\left[\begin{array}{cc} 2& \overline{0} \\0
&\overline{0}\end{array}\right]\in R$,
$r_R(y)=\left[\begin{array}{cc} 0 & \Bbb Z_2 \\0 &\Bbb
Z_2\end{array}\right]$ and it is not a direct summand of $R$ as a
right ideal. This implies that $R$ is not a right Rickart ring,
and so it is not right Goldie Rickart because of right
nonsingularity of $R$. }
\end{ex}

\begin{rem}{\rm
Clearly, every left (right) Rickart ring is left (right)
nonsingular, and so it is a left (right) Goldie Rickart ring. But
there is a left (right) Goldie Rickart ring which is not left
(right) Rickart. For instance, in Example \ref{left-right}, the
ring $R$ is left Goldie Rickart but not left Rickart. It is
obvious that a ring is left (right) Rickart if and only if it is
left (right) Goldie Rickart and left (right) nonsingular.}
\end{rem}

As in the following example, the Goldie Rickart property does not
pass on from a module to any its over module in general.

\begin{ex}{\rm Consider the ring $R$ in Example \ref{left-right} and let $M$ denote the right $R$-module $R$
and $N$ the submodule $\left[\begin{array}{cc} 0 & 0
\\0 &\Bbb Z_2\end{array}\right]$ of $M$. Note that $N=\left[\begin{array}{cc} 0 & 0\\0
&\overline{1}\end{array}\right]M$   for an idempotent
$\left[\begin{array}{cc} 0 & 0 \\0
&\overline{1}\end{array}\right]\in S$. Since  $N$ is a simple
module, obviously it is Rickart. Being $N$ nonsingular, it is also
Goldie Rickart. But it is known from Example \ref{left-right}, $M$
is not Goldie Rickart.  }
\end{ex}

According to Proposition \ref{dik toplan}, we have the next
result.
\begin{prop} Let $R$ be a right Goldie Rickart ring. Then for
every idempotent $e$ of $R$,  $eR$ is a Goldie Rickart module.
\end{prop}

\begin{lem}\label{fgproj} Every finitely generated projective module over a von Neumann regular
ring is Goldie Rickart.
\end{lem}
\begin{proof} Let $R$ be a von Neumann regular ring and $M$  a finitely generated projective
$R$-module. Then $M$ is a direct summand of a finitely generated
free $R$-module $F$. We can see $F$ as $\bigoplus\limits_{i=1}^n
R_i$ where $n\in \Bbb N$ and $R_i=R$ for all $i=1, \dots, n$.
Since $R$ is von Neumann regular, it is a right Goldie Rickart
ring. Hence $F$ is also Goldie Rickart from Proposition
\ref{kendi-dik} and so is $M$ due to Proposition \ref{dik toplan}.
\end{proof}

\begin{prop} Every finitely presented module over a von Neumann regular
ring is Goldie Rickart.
\end{prop}
\begin{proof} Let $R$ be a von Neumann regular ring and $M$  a finitely
presented $R$-module. Then $M$ is a flat module. Since $M$ is
finitely presented, it is finitely generated and projective. Hence
Lemma \ref{fgproj} completes the proof.
\end{proof}

\begin{prop} Every abelian free module over a right Goldie Rickart ring is
Goldie Rickart.
\end{prop}
\begin{proof} Let $R$ be a right Goldie Rickart ring and
$F$ an abelian free $R$-module. Assume that $F$ is $\bigoplus
\limits_{i\in \mathcal{I}} R_i$ where $\mathcal{I}$ is any index
set and $R_i=R$ for all $i\in \mathcal{I}$. Being $R$ Goldie
Rickart as an $R$-module, $F$ is Goldie Rickart from Corollary
\ref{diktoplam}.
\end{proof}

The following theorem gives a characterization of right Goldie
Rickart rings in terms of Goldie Rickart modules.

\begin{thm} Let $R$ be a ring and consider the following
conditions.
\begin{enumerate}
\item[{\rm (1)}] Every $R$-module is Goldie Rickart.
\item[{\rm (2)}] Every nonsingular $R$-module is Rickart and
$Z_2(R_R)$ is a direct summand of $R$.
\item[{\rm (3)}] Every projective $R$-module is Goldie Rickart.
\item[{\rm (4)}] Every free $R$-module is Goldie Rickart.
\item[{\rm (5)}] $R$ is a right Goldie Rickart ring.
\item[{\rm (6)}] Every cyclic projective $R$-module is Goldie Rickart.
\end{enumerate}
Then {\rm (1) $\Rightarrow$ (2) $\Rightarrow$ (3)
$\Leftrightarrow$ (4) $\Rightarrow$ (5) $\Leftrightarrow$ (6)}.
\end{thm}

\begin{proof} (1) $\Rightarrow$ (2) is clear by Theorem
\ref{goldie}. (3) $\Rightarrow$ (4) $\Rightarrow$ (5) and (6)
$\Rightarrow$ (5) are obvious.

(2) $\Rightarrow$ (4) Let $F$ be a free module. By hypothesis
$Z_2(R_R)$ is a direct summand of $R$, and so $Z_2(F)$ is a direct
summand of $F$. Let $F=Z_2(F)\oplus L$ for some submodule $L$ of
$F$. Since $L$ is nonsingular, it is Rickart. Then $F$ is Goldie
Rickart due to Theorem \ref{goldie}.

(4) $\Rightarrow$ (3) Let $P$ be a projective module. Then there
exists a free module $F$ and a submodule $K$ of $F$ such that
$P\cong F/K$. Hence $K$ is a direct summand of $F$. Let $F=K\oplus
N$ for some submodule $N$ of $F$. By (4), $F$ is Goldie Rickart
and due to Proposition \ref{dik toplan}, $N$ is Goldie Rickart and
so is $P$.

(5) $\Rightarrow$ (6) Let $M$ be a cyclic projective $R$-module.
Then $M\cong I$ for some direct summand right ideal $I$ of $R$.
Since $R$ is right Goldie Rickart, by Proposition \ref{dik
toplan}, $I$ is Goldie Rickart and so is $M$.
\end{proof}

\section{Relatively Goldie Rickart Modules}
Example \ref{direct-sum-ornek} shows that a direct sum of Goldie
Rickart modules need not be Goldie Rickart. In this section we
define relatively Goldie Rickart property in order to investigate
when are direct sums of Goldie Rickart modules also Goldie
Rickart.

\begin{df}\label{relative def} Let $M$ and $N$ be $R$-modules. $M$ is called {\it
$N$-Goldie Rickart} (or {\it relatively Goldie Rickart to $N$}) if
for every homomorphism $f: M\rightarrow N$, $f^{-1}(Z_2(N))$ is a
direct summand of $M$.
\end{df}

Note that in Definition \ref{relative def}, $Z_2(M)\leq
f^{-1}(Z_2(N))$. It is evident that a module $M$ is Goldie Rickart
if and only if it is $M$-Goldie Rickart.

\begin{thm}\label{relative} Let $M$ and $N$ be $R$-modules. Then $M$ is $N$-Goldie
Rickart if and only if for any direct summand $M_1$ of $M$ and any
submodule $N_1$ of $N$, $M_1$ is $N_1$-Goldie Rickart.
\end{thm}
\begin{proof} Let $M_1$ be a direct summand of $M$, $N_1$ a
submodule of $N$ and $f: M_1\rightarrow N_1$ a homomorphism. Then
$M=M_1\oplus M_2$ for some submodule $M_2$ of $M$, and so $f\oplus
0_{\mid_{M_2}} : M\rightarrow N$, say $g=f\oplus 0_{\mid_{M_2}}$.
Since $M$ is $N$-Goldie Rickart, $g^{-1}(Z_2(M))$ is a direct
summand of $M$. Now let $m_1+m_2\in g^{-1}(Z_2(M))$. It follows
that $g(m_1+m_2)=fm_1\in Z_2(N)\cap N_1=Z_2(N_1)$, hence
$g^{-1}(Z_2(M))\subseteq f^{-1}(Z_2(M_1))\oplus M_2$. Also for any
$m_1+m_2\in f^{-1}(Z_2(M_1))\oplus M_2$, we have
$g(m_1+m_2)=fm_1\in Z_2(N_1)\subseteq Z_2(N)$, and so
$f^{-1}(Z_2(M_1))\oplus M_2\subseteq g^{-1}(Z_2(M))$. Thus
$g^{-1}(Z_2(M))=f^{-1}(Z_2(M_1))\oplus M_2$ is a direct summand of
$M$. This implies that $f^{-1}(Z_2(M_1))$ is a direct summand of
$M_1$. Therefore $M_1$ is $N_1$-Goldie Rickart. The rest is clear.
\end{proof}

\begin{cor} Let $M$ be a module. Then the following are
equivalent.
\begin{enumerate}
\item[{\rm (1)}] $M$ is Goldie Rickart.
\item[{\rm (2)}] For any direct summand $N$ of $M$ and any
submodule $K$ of $M$, $N$ is $K$-Goldie Rickart.
\item[{\rm (3)}] For any direct summands $N$ and $K$ of $M$ and
any $f:M\rightarrow K$, ${f_{\mid_N}}^{-1}(Z_2(N))$ is a direct
summand of $N$.
\end{enumerate}
\end{cor}

Recall that a module $M$ has {\it $C_2$ condition} if any
submodule $N$ of $M$ which is isomorphic to a direct summand of
$M$ is also a direct summand. Let $M$ and $N$ be $R$-modules. $M$
is called {\it $N$-$C_2$ (or relatively $C_2$ to $N$)} if any
submodule $K$ of $N$ which is isomorphic to a direct summand of
$M$ is a direct summand of $N$.  Hence $M$ has $C_2$ condition if
and only if  it is $M$-$C_2$. It is proved in \cite[Proposition
2.26]{LRR1} that for modules $M$ and $N$, $M$ is $N$-$C_2$ if and
only if for any direct summand $K$ of $M$ and any submodule $L$ of
$N$, $K$ is $L$-$C_2$.

\begin{thm}\label{iff}  Let $\{M_i\}_{i\in\mathcal{I}}$ be a class of $R$-modules where
$\mathcal{I} = \{1, 2,..., n\}$. Assume that $M_i$ is $M_j$-$C_2$
for all $i$, $j\in \mathcal{I}$. Then $\bigoplus \limits_{i\in
\mathcal{I}} M_i$ is a Goldie Rickart module if and only if $M_i$
is $M_j$-Goldie Rickart for all $i$, $j\in \mathcal{I}$.
\end{thm}
\begin{proof} The necessity is clear from Theorem \ref{relative}.
For the sufficiency, assume that $M_i$ is $M_j$-Goldie Rickart for
all $i$, $j\in \mathcal{I}$. Without loss of generality we may
assume $n=2$. Let $f = f_{11}+ f_{21}+f_{12}+f_{22}$ denote the
matrix representation of $f$ where $f_{ij}:M_j\rightarrow M_i$.
Then $M_1$ and  $M_2$ are Goldie Rickart, and so $M_1 =
Z_2(M_1)\oplus L_{11}$ and $M_2 = Z_2(M_2)\oplus L_{22}$. Note
that $L_{11}$ and $L_{22}$ are Rickart modules. Now for $f_{12}:
M_2\rightarrow M_1$, by assumption, $M_2 =
f^{-1}_{12}(Z_2(M_1))\oplus L_{12}$. Similarly, for $f_{21}:
M_1\rightarrow M_2$, by assumption, $M_1 =
f^{-1}_{21}(Z_2(M_2))\oplus L_{21}$. Since $Z_2(M_1)\leq
f^{-1}_{21}(Z_2(M_2))$ and $Z_2(M_2)\leq f^{-1}_{12}(Z_2(M_1))$,
we have $f^{-1}_{21}(Z_2(M_2)) = Z_2(M_1)\oplus
[f^{-1}_{21}(Z_2(M_2))\cap L_{11}]$ and $f^{-1}_{12}(Z_2(M_1)) =
Z_2(M_2)\oplus [f^{-1}_{12}(Z_2(M_1))\cap L_{22}]$. Then $M =
Z_2(M_1)\oplus Z_2(M_2)\oplus [f^{-1}_{21}(Z_2(M_2))\cap
L_{11}]\oplus [f^{-1}_{12}(Z_2(M_1))\cap L_{22}]\oplus
L_{12}\oplus L_{21}$. Note that $Z_2(M) = Z_2(M_1)\oplus Z_2(M_2)$
and so $[f^{-1}_{21}(Z_2(M_2))\cap L_{11}]\oplus
[f^{-1}_{12}(Z_2(M_1))\cap L_{22}]\oplus L_{12}\oplus L_{21}$ is
nonsingular. $[f^{-1}_{21}(Z_2(M_2))\cap L_{11}]\oplus L_{21}$ and
$[f^{-1}_{12}(Z_2(M_1))\cap L_{22}]\oplus L_{12}$ are relatively
$C_2$. Also $[f^{-1}_{21}(Z_2(M_2))\cap L_{11}]\oplus L_{21}$ and
$[f^{-1}_{12}(Z_2(M_1))\cap L_{22}]\oplus L_{12}$ are relatively
Goldie Rickart since they are direct summands of relatively Goldie
Rickart modules $M_1$ and $M_2$, respectively. By \cite[Theorem
2.29]{LRR1}, $[f^{-1}_{21}(Z_2(M_2))\cap L_{11}]\oplus
[f^{-1}_{12}(Z_2(M_1))\cap L_{22}]\oplus L_{12}\oplus L_{21}$ is
Rickart. Therefore, by Theorem \ref{goldie}, $M$ is Goldie
Rickart.
\end{proof}

It is well known that every module which its endomorphism ring is
von Neumann regular has $C_2$ condition. In the light of Theorem
\ref{iff}, we can weaken the von Neumann regular endomorphism ring
condition in Proposition \ref{kendi-dik} as in the following.

\begin{cor}\label{direct sum} Let $M$ be a Goldie Rickart module with $C_2$ condition. Then any finite
direct sum of copies of $M$ is also Goldie Rickart.
\end{cor}

\begin{prop}\label{fgfree} Let $R$ be a right Goldie Rickart ring with $C_2$ condition as an $R$-module.
Then the following hold.
\begin{enumerate}
\item[{\rm (1)}] Every finitely generated free $R$-module is Goldie Rickart.
\item[{\rm (2)}] Every finitely generated projective $R$-module is Goldie
Rickart.
\end{enumerate}
\end{prop}
\begin{proof} (1) Clear from Corollary \ref{direct sum}.

(2) The condition (1) and Proposition \ref{dik toplan} complete
the proof.
\end{proof}

\begin{prop}\label{SIP-relative} Let $\{M_i\}_{i\in\mathcal{I}}$ be a class of $R$-modules
for an index set $\mathcal{I}$ and $N$ an $R$-module. Then the
following hold.\\
{\rm (1)} If $N$ has the summand intersection property for direct
summands which contain $Z_2(N)$ and $\mathcal{I}$ is finite, then
$N$ is $\bigoplus \limits_{i\in \mathcal{I}} M_i$-Goldie Rickart
if and only if it is $M_i$-Goldie Rickart for
all $i\in I$.\\
{\rm (2)} If $N$ has the strong summand intersection property for
direct summands which contain $Z_2(N)$, then $N$ is $\bigoplus
\limits_{i\in \mathcal{I}} M_i$-Goldie Rickart if and only if it
is $M_i$-Goldie Rickart for all $i\in I$ where $\mathcal{I}$ is arbitrary. \\
{\rm (3)} If $N$ has the strong summand intersection property for
direct summands which contain $Z_2(N)$, then $N$ is $\prod
\limits_{i\in \mathcal{I}} M_i$-Goldie Rickart if and only if it
is $M_i$-Goldie Rickart for all $i\in I$ where $\mathcal{I}$ is
arbitrary.
\end{prop}
\begin{proof} (1) Let $\mathcal{I} = \{1, 2,..., n\}$. The necessity is clear from Theorem \ref{relative}. For the
sufficiency, assume that $N$ is $M_i$-Goldie Rickart for all $i\in
\mathcal{I}$ and $f\in $ Hom$_R(N, \bigoplus \limits_{i\in
\mathcal{I}} M_i)$. Let $\pi_i$ denote the natural projection form
$\bigoplus \limits_{i\in \mathcal{I}} M_i$ to $M_i$ for every
$i\in \mathcal{I}$. Then $f=(\pi_1f, \dots , \pi_nf)$. It can be
shown that $f^{-1}(Z_2(\bigoplus \limits_{i\in \mathcal{I}}
M_i))=\bigcap\limits_{i\in I} (\pi_if)^{-1}(Z_2(M_i))$. By
assumption, $(\pi_if)^{-1}(Z_2(M_i))$ is a direct summand of $N$
for each $i\in \mathcal{I}$. Since $(\pi_if)^{-1}(Z_2(M_i))$
contains $Z_2(N)$ for each $i\in \mathcal{I}$, by hypothesis,
$f^{-1}(Z_2(\bigoplus \limits_{i\in \mathcal{I}} M_i))$ is a
direct summand of $N$, as desired.

(2) and (3) are proved similar to (1).
\end{proof}

\begin{cor}\label{ab} Let $\{M_i\}_{i\in\mathcal{I}}$ be a class of $R$-modules
where $\mathcal{I} = \{1, 2,..., n\}$. Then for every $j\in
\mathcal{I}$, $M_j$ is $\bigoplus \limits_{i\in \mathcal{I}}
M_i$-Goldie Rickart if and only if it is $M_i$-Goldie Rickart for
all $i\in \mathcal{I}$.
\end{cor}
\begin{proof} The necessity is clear from Theorem \ref{relative}.
For the sufficiency, let $j\in \mathcal{I}$ and $M_j$ be an
$M_i$-Goldie Rickart module for all $i\in \mathcal{I}$. Then $M_j$
is Goldie Rickart, and so it has the summand intersection property
for direct summands which contain $Z_2(M_j)$ by Proposition
\ref{sumint}. Thus the rest is clear from Proposition
\ref{SIP-relative}(1).
\end{proof}

\begin{thm}\label{relative inj}  Let $\{M_i\}_{i\in\mathcal{I}}$ be a class of $R$-modules where
$\mathcal{I} = \{1, 2,..., n\}$. Assume that $M_i$ is
$M_j$-injective for all $i<j\in \mathcal{I}$. Then for any
$R$-module $N$, $\bigoplus \limits_{i\in \mathcal{I}} M_i$ is an
$N$-Goldie Rickart module if and only if $M_i$ is $N$-Goldie
Rickart for all $i\in \mathcal{I}$.
\end{thm}
\begin{proof} Let $N$ be an $R$-module. The necessity is clear from Theorem \ref{relative}.
For the sufficiency, assume $M_i$ is $N$-Goldie Rickart for all
$i\in \mathcal{I}$. Suppose that $n=2$ and $f\in $
Hom$_R(M_1\oplus M_2, N)$. Then $f\iota_{M_i}\in $ Hom$_R(M_i, N)$
where $\iota_{M_i}:M_i\rightarrow M_1\oplus M_2$ is the inclusion
map, say $f_i=f\iota_{M_i}$ for $i=1, 2$. Hence
$M_1=f_1^{-1}(Z_2(N))\oplus L_1$ and $M_2=f_2^{-1}(Z_2(N))\oplus
L_2$ for some submodules $L_1$ of $M_1$ and $L_2$ of $M_2$. Since
$f_1^{-1}(Z_2(N))\oplus f_2^{-1}(Z_2(N)) \subseteq
f^{-1}(Z_2(N))$, we have $f^{-1}(Z_2(N))=f_1^{-1}(Z_2(N))\oplus
f_2^{-1}(Z_2(N))\oplus [f^{-1}(Z_2(N))\cap (L_1\oplus L_2)]$.  By
Theorem \ref{relative}, $L_i$ is $N$-Goldie Rickart for $i=1, 2$.
Note that $f_{\mid_{L_1\oplus
L_2}}^{-1}(Z_2(N))=f^{-1}(Z_2(N))\cap (L_1\oplus L_2)$. Since
$L_1$ is $L_2$-injective and $L_1\cap [f^{-1}(Z_2(N))\cap
(L_1\oplus L_2)]=0$, there exists a submodule $K$ of $L_1\oplus
L_2$ such that $L_1\oplus L_2=L_1\oplus K$ and $f^{-1}(Z_2(N))\cap
(L_1\oplus L_2)\subseteq K$. Clearly,
$f_{\mid_{K}}^{-1}(Z_2(N))=f^{-1}(Z_2(N))\cap (L_1\oplus L_2)$.
Also since $K\cong L_2$,  $K$ is $N$-Goldie Rickart. Thus
$f_{\mid_{K}}^{-1}(Z_2(N))$ is a direct summand of $K$. It follows
that $f^{-1}(Z_2(N))\cap (L_1\oplus L_2)$ is a direct summand of
$L_1\oplus L_2$. This implies that $f^{-1}(Z_2(N))$ is a direct
summand of $M_1\oplus M_2$. So $M_1\oplus M_2$ is $N$-Goldie
Rickart. Now suppose that $\bigoplus \limits_{i=1}^{n-1} M_i$ is
$N$-Goldie Rickart and we show  $\bigoplus \limits_{i=1}^{n} M_i$
is also $N$-Goldie Rickart. Since $\bigoplus \limits_{i=1}^{n-1}
M_i$ is $M_n$-injective and $M_n$ is $N$-Goldie Rickart, by the
preceding discussion, $\bigoplus \limits_{i=1}^{n} M_i$ is an
$N$-Goldie Rickart module. This completes the proof by the
induction on $n$.
\end{proof}

We conclude this paper by presenting a result on Goldie Rickart
property of $\bigoplus \limits_{i\in \mathcal{I}} M_i$ apart from
Theorem \ref{iff}.

\begin{cor} Let $\{M_i\}_{i\in\mathcal{I}}$ be a class of $R$-modules where
$\mathcal{I} = \{1, 2,..., n\}$. Assume that $M_i$ is
$M_j$-injective for all $i<j\in \mathcal{I}$. Then $\bigoplus
\limits_{i\in \mathcal{I}} M_i$ is a Goldie Rickart module if and
only if $M_i$ is $M_j$-Goldie Rickart for all $i, j\in
\mathcal{I}$.
\end{cor}
\begin{proof} The necessity is true by Theorem \ref{relative}.
For the sufficiency, assume that $M_i$ is $M_j$-Goldie Rickart for
all $i, j \in \mathcal{I}$. Due to Corollary \ref{ab}, $M_i$ is
$\bigoplus \limits_{i\in \mathcal{I}} M_i$-Goldie Rickart.
According to Theorem \ref{relative inj}, $\bigoplus \limits_{i\in
\mathcal{I}} M_i$ is also  $\bigoplus \limits_{i\in \mathcal{I}}
M_i$-Goldie Rickart, hence it is Goldie Rickart.
\end{proof}


\begin{thebibliography}{99}
\bibitem{AF} F. W. Anderson and K. R. Fuller, Rings and Categories of Modules, Springer-Verlag, New York, 1992.
\bibitem{AH} Sh. Asgari and A. Haghany, {\it t-Extending Modules and t-Baer Modules}, Comm. Algebra 39(2011), 1605-1623.
\bibitem{G} K. R. Goodearl, Singular Torsion and The Splitting
Properties, American Mathematical Society, 124, 1972.
\bibitem{Hat}  A. Hattori, {\it A Foundation of The Torsion Theory Over General
Rings}, Nagoya Math. J. 17(1960), 147-158.
\bibitem{L} T. Y. Lam, Lectures on Modules and Rings,
Springer-Verlag, New York, 1999.
\bibitem{LRR}  G. Lee, S. T. Rizvi and C. S. Roman, {\it Rickart Modules},  Comm. Algebra  38(11)2010, 4005-4027.
\bibitem{LRR1}  G. Lee, S. T. Rizvi and C. S. Roman, {\it Direct Sums of Rickart Modules},  J. Algebra 353(2012),
62-78.
\bibitem{M} S. Maeda, {\it On a Ring Whose Prncipal Right Ideals
Generated By Idempotents Form a Lattice}, J. Sci. Hiroshima Univ.
Ser. A 24(1960), 509-525.
\bibitem{NZ} W. K. Nicholson and Y. Zhou, {\it Strong Lifting}, J.
Algebra 285(2005), 795-818.
\bibitem{YZ} M. F. Yousif and Y. Zhou, {\it Semiregular, Semiperfect and Perfect Rings Relative to an Ideal},
Rocky Mountain J. Math. 32(4)(2002), 1651-1671.
\bibitem{Wa} R. Ware, {\it Endomorphism Rings of Projective Modules}, Trans. Amer. Math. Soc. 155(1971),
233-256.
\end{thebibliography}
\end{document}